\documentclass{article}[14pt,a4paper]

\usepackage[english]{babel}
\usepackage{amsthm}
\usepackage{amssymb}
\usepackage{amsfonts}
\usepackage{amsmath}
\usepackage{fancybox}
\usepackage[all,cmtip]{xy}
\usepackage{titletoc}
\usepackage[style=numeric,backend=bibtex]{biblatex}
\usepackage{bm}
\usepackage{verbatim}
\usepackage{csquotes}

\usepackage{mathrsfs}
\usepackage[usenames,dvipsnames]{xcolor}
\catcode`\ä = \active \catcode`\Ä = \active \catcode`\ö = \active \catcode`\Ö = \active \catcode`\ü = \active
\catcode`\Ü = \active

\defÄ{"A}
\defä{"a}
\defÖ{"O}
\defö{"o}
\defÜ{"U}
\defü{"u}

\newcommand{\BA}{{\mathbb{A}}}

\newcommand{\BC}{{\mathbb{C}}}

\newcommand{\BF}{{\mathbb{F}}}

\newcommand{\BL}{{\mathbb{L}}}

\newcommand{\BN}{{\mathbb{N}}}

\newcommand{\BV}{{\mathbb{V}}}

\newcommand{\BZ}{{\mathbb{Z}}}

\newcommand{\FC}{{\mathcal{C}}}

\newcommand{\FM}{{\mathcal{M}}}
\newcommand{\FN}{{\mathcal{N}}}

\newcommand{\FP}{{\mathcal{P}}}

\newcommand{\FV}{{\mathcal{V}}}

\newcommand{\bfv}{{\mathbf{v}}}
\newcommand{\bfw}{{\mathbf{w}}}
\newcommand{\bfl}{{\bm{\lambda}}}
\newcommand{\Hom}{{\text{Hom}}}
\newcommand{\g}{{\mathfrak{g}}}
\newcommand{\proj}{{\text{Proj}}}

\theoremstyle{plain}
\newtheorem{thm}{Theorem}[section]

\newtheorem{lemma}[thm]{Lemma}
\newtheorem{prop}[thm]{Proposition}

\theoremstyle{definition}
\newtheorem{defn}[thm]{Definition}

\newtheorem{rem}[thm]{Remark}

\renewcommand{\phi}{\varphi}

\newcommand{\GL}{\mathop{\rm GL}\nolimits}

\newcommand{\kvar}{\mathop{\rm KVar}\nolimits}
\newcommand{\kexp}{\mathop{\rm KExpVar}\nolimits}
\newcommand{\spec}{\mathop{\rm Spec}\nolimits}
\newcommand{\expm}{\mathscr{E}xp\mathscr{M}}
\newcommand{\lla}{\left\langle }
\newcommand{\rra}{\right\rangle}
\newcommand{\n}{{\text{nil}}}
\newcommand{\re}{{\text{reg}}}
\newcommand{\im}{{\text{Im}}}
\newcommand{\Ad}{{\text{Ad}}}
\newcommand{\tr}{{\text{tr}}}
\newcommand{\Id}{{\text{Id}}}

\title{Motivic classes of Nakajima quiver varieties} 
\author{Dimitri Wyss} 
\begin{document}

\thispagestyle{empty}
\maketitle

\begin{abstract}
\noindent We prove, that Hausel's formula for the number of rational points of a Nakajima quiver variety over a finite field also holds in a suitable localization of the Grothendieck ring of varieties. In order to generalize the arithmetic harmonic analysis in his proof we use Grothendieck rings with exponentials as introduced by Cluckers-Loeser and Hrushovski-Kazhdan.
\end{abstract}
\section{Introduction}

Let $\Gamma = (I,E,s,t)$ be a quiver, that is a finite vertex set $I$, a set of arrows $E\subset I\times I$ and maps $s,t:E\rightarrow I$ sending an arrow to its source and target. In \cite{nak94}\cite{nak98} Nakajima associates to $\Gamma$ and two dimension vectors $\bfv,\bfw \in \BN^I$ a smooth algebraic variety $\FM(\bfv,\bfw)$ called Nakajima quiver variety. A combinatorial formula for the Betti numbers of those varieties is proven in \cite{Hau2} using arithmetic methods. More precisely Hausel counts the number of rational points of these varieties over finite fields of large enough characteristic and then deduces their Betti numbers by a theorem of Katz \cite[Theorem 6.1.2.3]{HR08}.\\
	The main result of this article is Theorem \ref{mainth}, where we compute the class of $\FM(\bfv,\bfw)$ in a suitable localization $\mathscr{M}$ of the Grothendieck ring of varieties.	Explicitly let $\FP$ be the set of partitions. For $\lambda \in \FP$ we write $|\lambda|$ for its size and $m_k(\lambda)$ for the multiplicity of $k\in \BN$ in $\lambda$. Given any two partitions $\lambda,\lambda' \in \FP$ we define their inner product as $\lla \lambda,\lambda' \rra = \sum_{i,j\in \BN} \min(i,j)m_i(\lambda)m_j(\lambda').$ Then we prove
	
\begin{thm}\label{mainth} For a fixed dimension vector $\bfw \in \BN^I$  the motivic classes of the Nakajima quiver varieties $\FM_{\alpha,\chi}(\bfv,\bfw)$ in $\mathscr{M}$ are given by the generating function
	\begin{eqnarray}\label{for1} \sum_{\bfv \in \BN^I} [\FM(\bfv,\bfw)] \BL^{d_{\bfv,\bfw}}T^\bfv = \frac{\sum_{\bfl \in \FP^I} \frac{\prod_{e\in E} \BL^{\lla \lambda_{s(e)},\lambda_{t(e)}\rra} \prod_{i\in I} \BL^{\lla 1^{w_i},\lambda_i \rra} }{\prod_{i\in I} \BL^{\lla \lambda_i,\lambda_i \rra} \prod_k \prod_{j=1}^{m_k(\lambda_i)} (1-\BL^{-j})}T^{|\bfl|}}{\sum_{\bfl \in \FP^I} \frac{\prod_{e\in E} \BL^{\lla \lambda_{s(e)},\lambda_{t(e)}\rra}}{\prod_{i\in I} \BL^{\lla \lambda_i,\lambda_i \rra} \prod_k \prod_{j=1}^{m_k(\lambda_i)} (1-\BL^{-j})}T^{|\bfl|}},\end{eqnarray}
where $d_{\bfv,\bfw}$ denotes half the dimension of $\FM(\bfv,\bfw)$ and $\BL$ the class of the affine line in $\mathscr{M}$.
\end{thm}
This implies in particular, that $[\FM(\bfv,\bfw)]$ is given by a polynomial in $\BL$. \\	
The formula is the expected generalization of the count i.e. the cardinality $q$ of the finite field is simply replaced by $\BL$. However this generalization is not straightforward as we have to find a motivic analogue of the arithmetic harmonic analysis approach of \cite{Hau2}. We use the idea of \cite{CL}\cite{hk09} to 'add exponentials' to $\mathscr{M}$ in order to define a naive motivic Fourier transform. \\
The author was informed by Ben Davison and Sergey Mozgovoy, that they both can prove formula (\ref{for1}) using different methods.\\
	
\textit{Acknowledgments:} I would like to thank Tam\'{a}s Hausel for his guidance and helpful comments throughout the whole project and Bal\'{a}zs Szendr\"{o}i for his notes who started it. Furthermore I'm grateful to Fran\c{c}ois Loeser and Andrew Morrison for their explanations.\\ 
 	
\textit{Conventions:} Throughout the whole article we will work over an algebraically closed field $k$ of characteristic $0$. By a variety we mean a separated reduced scheme of finite type over $k$.

\section{Grothendieck rings with exponentials and a naive Fourier transform}

In this section we start by introducing various Grothendieck rings with exponentials following closely \cite{CLL}. This allows us to define a naive Fourier transform and prove a Fourier inversion formula for motivic functions. We should mention that nothing in this section is new, but rather a special case of the theory developed in \cite{CL}.\\
 The \textit{Grothendieck ring of varieties}, denoted by $\kvar$, is the quotient of the free abelian group generated by varieties modulo the relations
\[ X - Y\]
if $X$ and $Y$ are isomorphic and 
\[ X - Z - U,\] 
for $Z\subset X$ a closed subvariety and $U=X\setminus Z$. The multiplication is given by $[X]\cdot [Y] = [X\times Y]$, where we write $[X]$ for the class of a variety $X$ in $\kvar$. \\
The \textit{Grothendieck ring with exponentials} $\kexp$ is defined similarly. The generators are pairs $(X,f)$, where $X$ is a variety and $f:X \rightarrow \BA^1=\spec(k[T])$ a morphism. We impose three kinds of relations on the free abelian group generated by those pairs. 
\begin{enumerate}
\item[(i)] For two varieties $X,Y$, a morphism $f:X\rightarrow \BA^1$ and an isomorphism $u:Y\rightarrow X$ the relation
\[ (X,f) - (Y,u\circ f).\]
\item[(ii)] For a variety $X$, a morphism $f:X\rightarrow \BA^1$, a closed subvariety $Z\subset X$ and $U=X\setminus Z$ the relation
\[ (X,f) - (Z,f_{|Z})- (U,f_{|U}).\]
\item[(iii)] For a variety $X$ and $pr_{\BA^1}:X\times \BA^1\rightarrow \BA^1$ the projection the relation
\[ (X\times \BA^1,pr_{\BA^1}).\]
\end{enumerate} 
The class of $(X,f)$ in $\kexp$ will be denoted by $[X,f]$. We define the product of two generators $[X,f]$ and $[Y,g]$ as 
\[ [X,f]\cdot [X,g] = [X\times Y, f\circ pr_X + g\circ pr_Y],\]
where $f\circ pr_X + g\circ pr_Y: X\times Y \rightarrow \BA^1$ is the morphism sending $(x,y)$ to $f(x)+g(y)$. This gives $\kexp$ the structure of a commutative ring.\\
Denote by $\BL$ the class of $\BA^1$ resp. $(\BA^1,0)$ in $\kvar$ resp. $\kexp$. The localizations of $\kvar$ and $\kexp$ with respect to the the multiplicative subset generated by $\BL$ and $\BL^n-1$, where $n\geq 1$ are denoted by $\mathscr{M}$ and $\expm$.\\
For a variety $S$ there is a straight forward generalization of the above construction to obtain the \textit{relative Grothendieck rings} $\kvar_S,\kexp_S,\mathscr{M}_S$ and $\expm_S$. For example generators of $\kexp_S$ are pairs $(X,f)$ where $X$ is a $S$-variety (i.e. a variety with a morphism $X\rightarrow S$) and $f:X \rightarrow \BA^1$ a morphism. The class of $(X,f)$ in $\kexp_S$ will be denoted by $[X,f]_S$ or simply $[X,f]$ if the base variety $S$ is clear from the context. \\
There is a natural map
\begin{align*} \kvar_S &\rightarrow \kexp_S\\
							[X] &\mapsto [X,0]
\end{align*}
and similarly $\mathscr{M}_S \rightarrow \expm_S$, which are both injective ring homomorphisms by \cite[Lemma 1.1.3]{CLL}. Hence we don't need to distinguish between $[X]$ and $[X,0]$ for a $S$-variety $X$.\\
For a morphism of varieties $u: S \rightarrow T$ we have induced maps
\begin{align*} u_!&:\kexp_S \rightarrow \kexp_T, \ \ \ [X,f]_S \mapsto [X,f]_T\\
u^*&:\kexp_T \rightarrow \kexp_S, \ \ \ [X,f]_T \mapsto [X \times_T S, f \circ pr_X]_S. \end{align*}
In general $u^*$ is a morphism of rings and $u_!$ a morphism of additive groups. However it is straightforward to check that for any $u:S \rightarrow T$ and any $\phi\in \kexp_S$ we have
\begin{eqnarray}\label{genu} u_!(\BL \cdot\phi)=\BL\cdot u_!(\phi),
\end{eqnarray}
where $\BL$ denotes the class of $\BA^1\times S$ and $\BA^1\times T$ in $\kexp_S$ and $\kexp_T$ respectively.\\
Elements of $\kexp_S$ can be thought of as motivic functions on $S$. The evaluation of $\phi \in \kexp_S$ at a point $s: \spec(k) \rightarrow S$ is simply $s^*(\phi)\in \kexp_{\spec(k)}=\kexp$. Computations with these motivic functions can sometimes replace finite field computations. More precisely let $\BF_q$ be a finite field and fix a non-trivial additive character $\psi:\BF_q \rightarrow \BC^\times$. Assume that $S$, $X\rightarrow S$ and $f:X \rightarrow \BA^1$ are also defined over $\BF_q$. Then the class of $(X,f) \in \kexp_S$ corresponds to the function 

\[S(\BF_q) \rightarrow \BC, \ \ \ s \mapsto \sum_{x \in X_s(\BF_q)} \psi(f(x)).\]
Furthermore for a morphism $u: S \rightarrow T$ the operations $u_!$ and $u^*$ correspond to summation over the fibres of $u$ and composition with $u$ respectively.\\
An important identity for computing character sums over finite fields is 

\begin{eqnarray*} \sum_{v\in V} \psi(f(v)) = \begin{cases} q^{\dim(V)} &\text{ if } f=0\\
0 &\text{ else,}\end{cases}\end{eqnarray*}
where $V$ is a $\BF_q$ vector space and $f\in V^*$ a linear form.\\
To establish a similar identity in the motivic setting we let $V$ be a finite dimensional vector space over $k$ and $S$ a variety. We replace the linear form above with a family of affine linear forms i.e. a morphism $g=(g_1,g_2):X \rightarrow V^*\times k$, where $X$ is a $S$-variety. Then we define $f$ to be the morphism
\begin{align*} f:X\times V &\rightarrow k 	\\
								(x,v) &\mapsto \left\langle g_1(x),v\right\rangle + g_2(x).
\end{align*}
Finally we put $Z= g_1^{-1}(0)$. 

\begin{lemma}\label{orth} With the notation above we have the relation
\[ [X\times V,f] = \BL^{\dim V}[Z,{g_2}_{|Z}] \]
in $\kexp_S$.
\end{lemma}
\begin{proof} By using (\ref{genu}) we may assume $S=X$. Now because of \cite[Lemma 1.1.8]{CLL} it is enough to check for each point $x\in X$ the identity 
\[x^*([X\times V,f]) = x^*(\BL^{\dim V}[Z,{g_2}_{|Z}])\] and this is exactly Lemma $1.1.11$ of \textit{loc. cit.}
\end{proof}


Now we're ready to define a \textit{naive motivic Fourier transform} for functions on a finite dimensional $k$-vectorspace $V$ and prove an inversion formula. All of this is a special case of \cite[Section 7.1]{CL}.

\begin{defn} Let $p_V:V\times V^* \rightarrow V$ and $p_{V^*}:V\times V^* \rightarrow V^*$ be the obvious projections. 
\textit{The naive Fourier transformation} $\mathcal{F}_V$ is defined as 
\begin{align*} \mathcal{F}_V:\kexp_{V} &\rightarrow \kexp_{V^*} \\
								\phi &\mapsto p_{V^*!}(p_V^*\phi \cdot[V\times V^*,\lla,\rra]). \end{align*}
								
Here $\left\langle ,\right\rangle:V\times V^* \rightarrow k$ denotes the natural pairing.
\end{defn}
We will often write $\mathcal{F}$ instead of $\mathcal{F}_V$ when there's no ambiguity.\\
Notice that $\mathcal{F}$ is a homomorphism of groups and thus it is worth spelling out the definition in the case when $\phi = [X,f]$ is the class of a generator in $\kexp_V$. Letting $u:X\rightarrow V$ be the structure morphism we simply have 
\begin{eqnarray}\label{ftgen}
\mathcal{F}([X,f]) = [X \times V^*, f\circ pr_X+\left\langle u\circ pr_X,pr_{V^*}\right\rangle].
\end{eqnarray}

Now we're ready to prove an inversion formula for the naive Fourier transform.
\begin{prop}\label{finv} For every $\phi \in \kexp_{V}$  we have the identity
\[\mathcal{F}(\mathcal{F}(\phi)) = \BL^{\dim(V)} \cdot i^*(\phi),\]
where $i:V\rightarrow V$ is multiplication by $-1$.
\end{prop}


\begin{proof} Since $\mathcal{F}$ is a group homomorphism it is enough to prove the lemma for $\phi = [X,f]$ with $X\stackrel{u}{\rightarrow}V $.  Iterating (\ref{ftgen}) we get
\[\mathcal{F}(\mathcal{F}([X,f]))= [X\times V \times V^*, f\circ pr_X + \lla u\circ pr_X+pr_V,pr_{V^*}\rra].\]
Now we can apply Lemma \ref{orth} with $Z=\{ (x,v)\in X \times V\  |\ u(x) +v=0\}$ to obtain
\[[X\times V \times V^*, f\circ pr_X + \lla u\circ pr_X+pr_V,pr_{V^*}\rra] = \BL^{\dim V^*} [Z, f\circ pr_X].\]
Notice that $Z$ is a $V$-variety via projection onto the second factor and hence the projection onto the first factor induces a $V$-isomorphism $Z \cong (X\stackrel{i\circ u}{\rightarrow} V)$, which gives the desired result. 	
\end{proof}

\section{Motives of moment map equations}\label{msq}

The naive Fourier transform enables us to perform computations arising from the arithmetic harmonic analysis approach introduced in \cite{Hau1} in the motivic setting. In this section we prove a motivic version of the crucial Proposition 1 of \textit{loc. cit.} on the number of points of certain moment map fibers.\\
Let $G$ be a reductive algebraic group over $k$ with Lie algebra $\mathfrak{g}$ and $\rho:G \rightarrow\GL(V)$ a representation. The derivative of $\rho$ is the Lie algebra representation $\varrho:\mathfrak{g} \rightarrow \mathfrak{gl}_n$. We define the \textit{moment map}
\begin{align*} \mu:V \times V^* \rightarrow \mathfrak{g}^*
\end{align*}
for $(v,w) \in V\times V^*$ and $X \in \mathfrak{g}$ by the formula
\begin{eqnarray}\label{mmap} \left\langle \mu(v,w),X \right\rangle = \left\langle  \varrho(X)(v),w \right\rangle.
\end{eqnarray}

Our goal is now to compute for $\xi \in \mathfrak{g}^*$ the motive of $\mu^{-1}(\xi)$ in $\expm$.

\begin{rem} Notice that $\rho$ induces an action of $G$ on the symplectic vector space $V\times V^*$ by the formula
\[ g\cdot (v,w) = (\rho(g)v,\rho(g^{-1})^*w)\]
and if $k=\BC$ one can check that $\mu$ is indeed a moment map for this action.
\end{rem}

We define
\[a_{\varrho} = \left\{ (v,X) \in V \times \mathfrak{g} \ |\ \varrho(X)v=0 \right\},\]
which is a $\g$-variety via he projection onto the second factor $\pi:a_\varrho \rightarrow \g$. Analogous to \cite[Proposition 1]{Hau1} we have

\begin{prop}\label{prop1} For any $\xi \in \g^*$ the identity
\[ [\mu^{-1}(\xi)] = \BL^{\dim V- \dim\g}[a_\varrho,\lla -\pi,\xi\rra]\]
holds in $\expm$. 
\end{prop}

\begin{proof} 
We consider $V \times V^*$ as a $\g^*$-variety via the moment map $\mu$. Then by (\ref{ftgen}) the naive Fourier transform of its class in $\kexp_{\g^*}$ is 
\begin{align*} \mathcal{F}([V\times V^*]) = [V \times V^* \times \g, \lla \mu\circ pr_{V\times V^*}, pr_\g\rra]. 
\end{align*}
Now by the definition (\ref{mmap}) of $\mu$ we have
\[ [V \times V^* \times \g, \lla \mu\circ pr_{V\times V^*}, pr_\g\rra] =  [V \times V^* \times \g, \lla(\varrho \circ pr_\g)pr_V,pr_{V^*}\rra].\]
Thus lemma \ref{orth} with $X= V\times \g$ and $Z = a_\varrho$ gives
\[\mathcal{F}([V\times V^*]) = \BL^{\dim V} [a_\varrho].\]
Next we apply $\mathcal{F}$ again and use the inversion lemma \ref{finv} to get
\[ \BL^{\dim \g} i^*[V\times V^*] = \mathcal{F}(\BL^{\dim V} [a_\varrho]) = \BL^{\dim V} [a_\varrho \times \g^*,\lla \pi\circ pr_{a_\varrho},pr_{\g^*}\rra]. \]
Finally passing to $\expm_{\g^*}$ to invert $\BL^{\dim \g}$ and using $(i^*)^2= \Id_{\g^*}$ gives
\[[V\times V^*] = \BL^{\dim V-\dim \g}[a_\varrho \times \g^*,\lla -\pi\circ pr_{a_\varrho},pr_{\g^*}\rra].\]
The result now follows from pulling back both sides allong $\xi:\spec(k) \rightarrow \g^*$.
\end{proof}

\section{Nakajima Quiver varieties}

In this section we recall the definition of Nakajima quiver varieties. Almost everything can be found in more detail in \cite{Hau2} or in the original sources \cite{nak94}\cite{nak98}.\\
Let $\Gamma = (I,E)$ be a quiver with $I=\{1,2,\dots,n\}$ the set of vertices and $E$ the set of arrows. We denote by $s(e)$ and $t(e)$ the source and target vertex of an arrow $e\in E$. For each $i\in I$ we fix finite dimensional vector spaces $V_i, W_i$ and write $\bfv=(\dim V_i)_{i\in I}, \bfw=(\dim W_i)_{i\in I} \in \mathbb{N}^I$ for their dimension vectors. From this data we construct the vector space
\[ \BV_{\bfv,\bfw}= \bigoplus_{e\in E} \Hom(V_{s(e)},V_{t(e)}) \oplus \bigoplus_{i\in I} \Hom(W_i,V_i),\]
the algebraic group 
\[ G_\bfv = \prod_{i\in I} \GL(V_i)\]
and its Lie algebra
 \[ \mathfrak{g}_\bfv = \bigoplus_{i\in I} \mathfrak{gl}(V_i).\]
We have a natural representation
\[ \rho_{\bfv,\bfw}: G_\bfv \rightarrow \GL(\BV_{\bfv,\bfw})\]
and its derivative 
\[ \varrho_{\bfv,\bfw}: \mathfrak{g}_\bfv \rightarrow \mathfrak{gl}(\BV_{\bfv,\bfw}).\]
For $g=(g_i)_{i\in I}\ ,X=(X_i)_{i\in I} $ and $\phi=(\phi_e,\phi_i)_{e\in E,i\in I}\in \BV_{\bfv,\bfw}$ they are given by the formulas
\begin{align*} 
\rho_{\bfv,\bfw}(g)\phi &= (g_{t(e)}\phi_e g_{s(e)}^{-1}, g_i\phi_i)_{e\in E,i\in I}\\
\varrho_{\bfv,\bfw}(X)\phi &= (X_{t(e)}\phi_e - \phi_e X_{s(e)},X_i \phi_i)_{e\in E,i\in I}.
\end{align*}
Now we're exactly in the situation of section \ref{msq} i.e. $G_\bfv$ acts on the vector space $\BV_{\bfv,\bfw}\times	\BV_{\bfv,\bfw}^*$ with moment map  
\[ \mu_{\bfv,\bfw}: \BV_{\bfv,\bfw} \oplus \BV_{\bfv,\bfw}^* \rightarrow \mathfrak{g}_\bfv^*,\] 
given by (\ref{mmap}).\\
Next we fix $\alpha \in k$ and define the affine variety $\FV_\alpha(\bfv,\bfw)= \mu_{\bfv,\bfw}^{-1}(\alpha\mathbf{1}_\bfv)$, where $\mathbf{1}_\bfv\in \mathfrak{g}_\bfv^*$ is defined by $\mathbf{1}_\bfv(X)= \sum_{i\in I} \tr X_i$ for $X\in \mathfrak{g}_\bfv$. Following \cite{nak98} we set $\chi$ to be the character of $G_\bfv$ given by $\chi(g) = \prod_{i\in I} \det(g_i)^{-1}$. Furthermore we put
\[ k[\FV_\alpha(\bfv,\bfw)]^{G_\bfv,\chi^m}=\{ f\in k[\FV_\alpha(\bfv,\bfw)]\ |\ f(g(x))=\chi(g)^mf(x)\ \forall \	 x\in \FV_\alpha(\bfv,\bfw)\}.\]
Then $\bigoplus_{m\geq 0} k[\FV_\alpha(\bfv,\bfw)]^{G_\bfv,\chi^m}$ is a graded algebra and we define the \textit{Nakajima quiver variety} as 
\begin{eqnarray}\label{GiT} \FM_{\alpha,\chi}(\bfv,\bfw) = \proj\left(\bigoplus_{m\geq 0} k[\FV_\alpha(\bfv,\bfw)]^{G_\bfv,\chi^m}\right).\end{eqnarray}
In this article we will mostly be concerned with the affine version given by 
\[\FM_\alpha(\bfv,\bfw)= \spec(k[\FV_\alpha(\bfv,\bfw)]^{G_\bfv}).\]

A practical reason for this is the following fact.
\begin{lemma}\label{facts} 
The quotient $\FV_\alpha(\bfv,\bfw) \rightarrow \FM_\alpha(\bfv,\bfw)$ is a Zariski locally trivial $G_\bfv$-principal bundle.
\end{lemma}
\begin{proof} Notice that any $G_\bfv$-principal bundle is Zariski locally trivial, see \cite[Lemma 5 and 6]{serre58}. Hence it is enough to prove, that $\FV_\alpha(\bfv,\bfw) \rightarrow \FM_\alpha(\bfv,\bfw)$ is a $G_\bfv$-principal bundle. In view of Proposition 0.9 and Amplification 1.3 of \cite{MFK94} it is enough to show, that $G_\bfv$ acts scheme-theoretically freely on $\FV_\alpha(\bfv,\bfw)$ i.e. the natural map
\[ G_\bfv \times \FV_\alpha(\bfv,\bfw) \rightarrow \FV_\alpha(\bfv,\bfw) \times \FV_\alpha(\bfv,\bfw)\]
is a closed immersion, which can be done similarly to \cite[Lemma 6.5]{Re03}.
\end{proof}
Finally we notice, that motivically little is lost by restricting ourselves to $\FM_\alpha(\bfv,\bfw)$. Indeed, for $\alpha \in k^\times$ the affinization map 
\begin{eqnarray}\label{affi} \FM_{\alpha,\chi}(\bfv,\bfw) \rightarrow \FM_{\alpha}(\bfv,\bfw)\end{eqnarray}

is an isomorphism (cf. \cite[Lemma 7]{Hau2}) and for $\alpha =0$ we have
\begin{prop}\label{ats} The classes of $\FM_{0,\chi}(\bfv,\bfw)$ and $\FM_{1}(\bfv,\bfw)$ agree in $\mathscr{M}$.
\end{prop}
\begin{proof} The argument is similar to \cite[Theorem 8]{Hau2}. Let $\mu: \BV_{\bfv,\bfw} \oplus \BV_{\bfv,\bfw}^* \oplus k \rightarrow \mathfrak{g}_\bfv^*$ be the map given by $\mu(\phi,\psi,z)= \mu_{\bfv,\bfw}(\phi,\psi) - z\mathbf{1}_\bfv$. Letting $G_\bfv$ act trivially on $k$, the fiber $\FV = \mu^{-1}(0)$ is $G_\bfv$-invariant and we define the GIT quotient as in (\ref{GiT}) by
\[ \FN = \proj\left(\bigoplus_{m\geq 0} k[\FV]^{G_\bfv,\chi^m}\right).\]
We have a natural map $f:\FN\rightarrow k$ induced by the projection $\FV\rightarrow k$. Analogous to \cite[Corollary 3.12]{nak98} we deduce that $\FN$ is non-singular and furthermore $f^{-1}(\alpha) \cong \FM_{\alpha,\chi}(\bfv,\bfw)$ for every $\alpha \in k$ since $\proj$ is compatible with base change (cf. \cite[Remark 13.27]{go10}).\\
Now the $k^\times$-action on $\FV$ given by
\[ \lambda\cdot (\phi,\psi,z) = (\lambda\phi,\lambda\psi,\lambda^2z),\]
descends to an action on $\FN$ and we have an equality of fixpoint sets $\FN^{k^\times} = f^{-1}(0)^{k^\times} = \FM_{0,\chi}(\bfv,\bfw)^{k^\times}$. Hence the Bialynicki-Birula Theorem \cite[Theorem 4.1]{bi73} (notice that the existence of a $k^\times$-invariant quasi-affine open covering is automatic by \cite[Corollary 2]{su74}) implies $[\FN] = \BL [\FM_{0,\chi}(\bfv,\bfw)]$.\\
On the other hand using that $\mu_{\bfv,\bfw}$ is bilinear we obtain a trivialization $f^{-1}(k^\times) \cong \FM_{1,\chi}(\bfv,\bfw) \times k^\times$ and hence
\[[\FN]=[\FM_{0,\chi}(\bfv,\bfw)] + [f^{-1}(k^\times)]=[\FM_{0,\chi}(\bfv,\bfw)] + (\BL-1)\FM_{1}(\bfv,\bfw),\]
where we also used the isomorphism (\ref{affi}). Comparing the two expressions for $[\FN]$ implies the result.
\end{proof}

\section{The main computation}\label{tmc}

In this section we prove our main Theorem \ref{mainth}, a combinatorial formula for the motive of a Nakajima quiver variety $\FM_{\alpha,\chi}(\bfv,\bfw)$	. By (the proof of) Proposition \ref{ats} it is enough to consider the following generating series
\begin{eqnarray}\label{defin}\Phi(\bfw) = \sum_{\bfv \in \BN^I} [\FM_1(\bfv,\bfw)] \BL^{d_{\bfv,\bfw}}T^\bfv \in \mathscr{M}[[T_1,\dots,T_n]], \end{eqnarray}
where we put $d_{\bfv,\bfw} = \dim(\mathfrak{g}_\bfv)-\dim(\BV_{\bfv,\bfw})$.
Having proposition \ref{prop1} available, we can argue along the lines of the finite field computations in \cite{Hau2}, with one difference. Namely, given a fibration $f:X \rightarrow Y$ with fiber $F$ we cannot deduce in general 
\begin{eqnarray}\label{trif} [X]=[F][Y] \end{eqnarray}
in $\kvar$ or $\mathscr{M}$, whereas a similar relation clearly holds over a finite field. However (\ref{trif}) holds if the fibration is Zariski-locally trivial i.e. $Y$ admits an open covering $Y=\cup_{j} U_j$ such that $f^{-1}(U_j) \cong F\times U_j$. Indeed, in this case we have 
\[[X] = \sum_{j} [f^{-1}(U_j)] - \sum_{j_1 < j_2} [f^{-1}(U_{j_1} \cap U_{j_2})] +... = [F][Y].\]
Combining this with Lemma \ref{facts} and Proposition \ref{prop1} we get
\begin{eqnarray}\label{phisim} \Phi(\bfw) = \sum_{\bfv \in \BN^I} \frac{[\FV_1(\bfv,\bfw)]}{[\GL_\bfv]} \BL^{d_{\bfv,\bfw}}T^\bfv =\sum_{\bfv \in \BN^I} \frac{[a_{\varrho_{\bfv,\bfw}},\lla-\pi, \mathbf{1}_\bfv\rra]}{[\GL_\bfv]}T^\bfv, 
\end{eqnarray}
with the notations 
\[a_{\varrho_{\bfv,\bfw}} = \{ (\phi,X) \in \BV_{\bfv,\bfw} \times \mathfrak{g}_\bfv \ | \ \varrho_{\bfv,\bfw}(X)v=0\}\]
 and $\pi:a_{\varrho_{\bfv,\bfw}}\rightarrow \mathfrak{g}_\bfv$ the natural projection.\\
Next we use some basic linear algebra to split up the above generating series into a regular and a nilpotent part. Given a finite dimensional vector space $V$ of dimension $n$ and an endomorphism $X$ of $V$, we can write $V=N(X) \oplus R(X)$, where $N(X) = \ker(X^n)$ and $R(X)=\im(X^n)$. With respect to this decomposition we have $X=X^\n \oplus X^\re$ with $X^\n= X_{|N(X)}$ nilpotent and $X^\re=X_{|R(X)}$ regular. \\
Now let $\bfv'= (v_i')_{i\in I}$ with $\bfv' \leq \bfv $ (i.e the inequality holds for every entry). We define the three varieties
\begin{align*}
a_{\varrho_{\bfv,\bfw}}^{\bfv'} &= \{ (\phi,X) \in a_{\varrho_{\bfv,\bfw}} \ | \ \dim(N(X_i))=v_i' \text{ for }i\in I\} \\
 a_{\varrho_{\bfv,\bfw}}^{\n} &= \{ (\phi,X) \in a_{\varrho_{\bfv,\bfw}} \ | \ X \text{ nilpotent}\}, \\
 a_{\varrho_{\bfv,\bfw}}^\re &= \{ (\phi,X) \in a_{\varrho_{\bfv,\bfw}}\ | \ X \text{ regular}\}.
\end{align*}

\begin{lemma}\label{first} For every $\bfv' \leq \bfv $ we have the following relation in $\expm$
\begin{eqnarray}\label{msp1}  \frac{[a_{\varrho_{\bfv,\bfw}}^{\bfv'},\lla-\pi, \mathbf{1}_\bfv\rra]}{[G_\bfv]} = \frac{[a_{\varrho_{\bfv',\bfw}}^{\n}]}{[G_{\bfv'}]} \frac{[a_{\varrho_{\bfv-\bfv',0}}^\re, \lla-\pi, \mathbf{1}_{\bfv-\bfv'}\rra]}{[G_{\bfv-\bfv'}]}.\end{eqnarray}
\end{lemma}

\begin{proof}  Fix for all $i\in I$ a decomposition $V_i=V_i' \oplus V_i''$ with $\dim(V_i')= v'_i$. This induces inclusions
\[\BV_{\bfv',\bfw}\oplus \BV_{\bfv-\bfv',0} \hookrightarrow \BV_{\bfv,\bfw} \text{  and  } \mathfrak{g}_{\bfv'} \oplus \mathfrak{g}_{\bfv-\bfv'} \hookrightarrow \mathfrak{g}_{\bfv}.\]
We will prove that the morphism
\begin{align*} \Delta: a_{\varrho_{\bfv',\bfw}}^{\n} \times a_{\varrho_{\bfv-\bfv',0}}^\re \times G_\bfv &\rightarrow a_{\varrho_{\bfv,\bfw}}^{\bfv'}\\ 
							(\phi',X',\phi'',X'',g) &\mapsto (\rho_{\bfv,\bfw}(g)(\phi'\oplus \phi''), \Ad_g(X' \oplus X''))
\end{align*}

is a Zariski-locally trivial $G_{\bfv'} \times G_{\bfv-\bfv'}$-fibration. Since for every $(\phi',X') \in a_{\varrho_{\bfv',\bfw}}^{\n}$ we have
\[ \lla-\pi, \mathbf{1}_{\bfv'}\rra(\phi',X')= \sum_{i\in I} \tr X'_i=0,\]
this will imply the lemma using (\ref{trif}).\\
First notice that $\Delta$ is well defined because
\[\varrho_{\bfv,\bfw} \circ \Ad_g = \Ad_{\rho_{\bfv,\bfw}(g)} \circ \varrho_{\bfv,\bfw}.\]
The $G_{\bfv'} \times G_{\bfv-\bfv'}$-action on the domain of $\Delta$ is given as follows. 
For \linebreak $h=(h',h'')\in G_{\bfv'}\times G_{\bfv-\bfv'}$ and $(\phi',X',\phi'',X'',g)\in a_{\varrho_{\bfv',\bfw}}^{\n} \times a_{\varrho_{\bfv-\bfv',0}}^\re \times G_\bfv$  we set
\[ h\cdot (\phi',X',\phi'',X'',g) = (\rho_{\bfv',\bfw}(h')\phi',\Ad_{h'}X',\rho_{\bfv-\bfv',0}(h'')\phi'',\Ad_{h''}X'',gh^{-1}),\]
where $gh^{-1}$ is understood via the inclusion  $G_{\bfv'} \times G_{\bfv-\bfv'} \hookrightarrow G_{\bfv}$. One checks directly that $\Delta$ is invariant under this action and hence each fiber of $\Delta$ carries a free $G_{\bfv'}\times G_{\bfv-\bfv'}$-action. \\
On the other hand, assume $\Delta(\phi_1',X_1',\phi_1'',X_1'',g_1) = \Delta(\phi_2',X_2',\phi_2'',X_2'',g_2)$. This implies
\[ \Ad_{g_2^{-1}g_1}(X_1' \oplus X_1'')=X_2'\oplus X_2''.\]
Since $X_j'$ is nilpotent and $X_j''$ regular for $j=1,2$, the decomposition $V_i=V_i' \oplus V_i''$ is preserved by $g_2^{-1}g_1$ i.e. $g_2^{-1}g_1 \in G_{\bfv'}\times G_{\bfv-\bfv'}$, which shows that each fiber of $\Delta$ is isomorphic to $G_{\bfv'}\times G_{\bfv-\bfv'}$.\\ 
Finally to trivialize $\Delta$ locally we notice, that there is an open covering $a_{\varrho_{\bfv,\bfw}}^{\bfv'}= \cup_j U_j$ and algebraic morphisms $t_j:U_j \rightarrow G_\bfv$ such that for $X\in U_j$ and $i\in I$ the columns of the matrix $t_j(X)_i$ form a basis of $N(X_i)$ and $R(X_i)$.
\end{proof}
Now we use the stratification $a_{\varrho_{\bfv,\bfw}}= \coprod_{\bfv' \leq \bfv} a_{\varrho_{\bfv,\bfw}}^{\bfv'}$ together with lemma \ref{first} to get

\begin{align}\label{aster} \nonumber \Phi(\bfw) &\stackrel{(\ref{phisim})}{=} \sum_{\bfv \in \BN^I} \frac{[a_{\varrho_{\bfv,\bfw}},\lla-\pi, \mathbf{1}_\bfv\rra]}{[\GL_\bfv]}T^\bfv\\
\nonumber &= \sum_{\bfv \in \BN^I}\sum_{\bfv' \leq \bfv} \frac{[a_{\varrho_{\bfv,\bfw}}^{\bfv'},\lla-\pi, \mathbf{1}_\bfv\rra]}{[\GL_\bfv]}T^\bfv\\ 
\nonumber &\stackrel{(\ref{msp1})}{=}  \sum_{\bfv \in \BN^I}\sum_{\bfv' \leq \bfv}\frac{[a_{\varrho_{\bfv',\bfw}}^{\n}]}{[G_{\bfv'}]} \frac{[a_{\varrho_{\bfv-\bfv',0}}^\re, \lla-\pi,\mathbf{1}_{\bfv-\bfv'}\rra]}{[G_{\bfv-\bfv'}]}T^\bfv \\
&= \Phi_{\n}(\bfw)\Phi_\re,
\end{align}
where we used the notations 
\[ \Phi_{\n}(\bfw) = \sum_{\bfv \in \BN^I} \frac{[a_{\varrho_{\bfv,\bfw}}^{\n}]}{[G_{\bfv}]}T^\bfv\]
and
\[ \Phi_\re = \sum_{\bfv \in \BN^I} \frac{[a_{\varrho_{\bfv,0}}^\re, \lla-\pi, \mathbf{1}_{\bfv}\rra]}{[G_{\bfv}]}T^\bfv.\]

Now \cite[Lemma 3]{Hau2} implies $\Phi(\mathbf{0}) = 1$ and therefore
\begin{eqnarray}\label{fract} \Phi(\bfw) = \frac{\Phi_{\n}(\bfw)}{\Phi_{\n}(0)},\end{eqnarray}	

which leaves us with computing $\Phi_{\n}(\bfw)$.\\
We denote by $\FP$ the set of all partitions $\lambda = (\lambda^1,\lambda^2,\dots)$, where $\lambda^1\geq \lambda^2\geq\dots$. The size of $\lambda$ is $|\lambda|=\lambda^1+\lambda^2+\dots$ and $\FP_n$ denotes the set of partitions of size $n$. For $\lambda \in \FP_n$ we write $\FC(\lambda)$ for the nilpotent conjugacy class, whose Jordan normal form is given by $\lambda$. For $\bfl = (\lambda_i)\in \FP^I$ with $\lambda_i\in \FP_{v_i}$ we set 

\[a_{\varrho_{\bfv,\bfw}}^{\n}(\bfl) = \{ (\phi,X) \in a_{\varrho_{\bfv,\bfw}}^{\n} \ |\ X_i \in \FC(\lambda_i) \},\]
which gives the stratification

\begin{eqnarray}\label{strat2}a_{\varrho_{\bfv,\bfw}}^{\n} = \coprod_{\substack{\bfl \in \FP^I \\ \lambda_i\in \FP_{v_i}}} a_{\varrho_{\bfv,\bfw}}^{\n}(\bfl).\end{eqnarray}

To compute $[a_{\varrho_{\bfv,\bfw}}^{\n}(\bfl)]$ we look at the projection 

\begin{eqnarray}\label{vb} \pi: a_{\varrho_{\bfv,\bfw}}^{\n}(\bfl) \rightarrow\FC(\bfl)= \prod_{i\in I} \FC(\lambda_i).\end{eqnarray}
The fiber of $\pi$ over $X\in \FC(\bfl)$ is simply $\ker(\varrho_{\bfv,\bfw}(X))$. Because of $\varrho_{\bfv,\bfw} \circ \Ad_g = \Ad_{\rho_{\bfv,\bfw}(g)} \circ \varrho_{\bfv,\bfw}$ the dimensions of those kernels are constant and hence $\pi$ is a vector bundle of rank, say, $\kappa_{\bfv,\bfw}(\bfl)$. 
\begin{lemma}\label{second} Denote by $Z(\bfl) \subset G_\bfv$ the centralizer of (some element in) $\FC(\bfl)$. We have the following relation in $\mathscr{M}$.
\begin{eqnarray}\label{idkn} \frac{[a_{\varrho_{\bfv,\bfw}}^{\n}(\bfl)]}{[G_\bfv]} = \frac{\BL^{\kappa_{\bfv,\bfw}(\bfl)}}{[Z(\bfl)]}\end{eqnarray}
\end{lemma}
\begin{proof}
The formula (\ref{zent}) below shows in particular that $[Z(\bfl)]$ is invertible in $\mathscr{M}$. Since the projection (\ref{vb}) is a vector bundle, we're left with proving $[G_\bfv]/[Z(\bfl)]= [\FC(\bfl)]$. Since $\FC(\bfl)$ is isomorphic to $G_\bfv / Z(\bfl)$, see for example \cite[Chapter 3.9.1]{bo12}, it is enough to prove that the $Z(\bfl)$-principal bundle $G_\bfv \rightarrow  G_\bfv / Z(\bfl)$ is Zariski locally trivial by (\ref{trif}). In fact, this is true for every $Z(\bfl)$-principal bundle, which follows from combining Propositions 3.13 and 3.16 of \cite{Mer13}.
\end{proof}

 
To compute $\kappa_{\bfv,\bfw}(\bfl)$ and $[Z(\bfl)]$, denote by $m_k(\lambda)$ the multiplicity of $k\in \BN$ in a partition $\lambda \in \FP$. Then given any two partitions $\lambda,\lambda' \in \FP$ we define their inner product to be
\[ \lla \lambda,\lambda' \rra = \sum_{i,j\in \BN} \min(i,j)m_i(\lambda)m_j(\lambda').\]
Lemma 3.3 in \cite{Hua} implies now
\begin{eqnarray}\label{kappa} \kappa_{\bfv,\bfw}(\bfl) = \sum_{e\in E} \lla  \lambda_{s(e)},\lambda_{t(e)}\rra + \sum_{i\in I} \lla 1^{w_i},\lambda_i \rra, 
\end{eqnarray}
where $1^{w_i} \in \FP_{w_i}$ denotes the partition $(1,1,\dots,1)$.\\
For $[Z(\bfl)]$ we can use the formula (1.6) from \cite[Chapter 2.1]{Mac}. There the formula is worked out over a finite field but Lemma 1.7 of \textit{loc. cit.} holds over any field. In our notation this gives (see \cite[Chapter 3]{Hua} for details)
\begin{eqnarray}\label{zent} [Z(\bfl)] =\prod_{i\in I} \BL^{\lla \lambda_i,\lambda_i \rra} \prod_{k\in \BN} \prod_{j=1}^{m_k(\lambda_i)} (1-\BL^{-j}).\end{eqnarray}

Finally combining (\ref{fract}), (\ref{strat2}), (\ref{idkn}), (\ref{kappa}) and (\ref{zent}) we obtain our main theorem \ref{mainth}.

\printbibliography

\end{document}